\documentclass[12pt]{amsart}
\usepackage{amsmath, amsfonts, amssymb, amsthm,epsfig}
\usepackage[utf8]{inputenc}
\usepackage[T1]{fontenc}
\usepackage[pdftex,bookmarks=true,bookmarksnumbered=true,breaklinks=true, colorlinks=true]{hyperref}	
\usepackage[ 	
capitalise 	
]{cleveref}	
\usepackage{enumerate,nicefrac}
\usepackage{tikz-cd}
\usepackage{adjustbox}

\usetikzlibrary{calc}
\usetikzlibrary{decorations.pathmorphing}

\tikzset{curve/.style={settings={#1},to path={(\tikztostart)
			.. controls ($(\tikztostart)!\pv{pos}!(\tikztotarget)!\pv{height}!270:(\tikztotarget)$)
			and ($(\tikztostart)!1-\pv{pos}!(\tikztotarget)!\pv{height}!270:(\tikztotarget)$)
			.. (\tikztotarget)\tikztonodes}},
	settings/.code={\tikzset{quiver/.cd,#1}
		\def\pv##1{\pgfkeysvalueof{/tikz/quiver/##1}}},
	quiver/.cd,pos/.initial=0.35,height/.initial=0}

\tikzset{tail reversed/.code={\pgfsetarrowsstart{tikzcd to}}}
\tikzset{2tail/.code={\pgfsetarrowsstart{Implies[reversed]}}}
\tikzset{2tail reversed/.code={\pgfsetarrowsstart{Implies}}}
\tikzset{no body/.style={/tikz/dash pattern=on 0 off 1mm}}

\let\BFseries\bfseries\def\bfseries{\BFseries\mathversion{bold}} 
\newcommand{\ind}{1\hspace{-0.098cm}\mathrm{l}}

\def\E{\mathbb{E}}

\theoremstyle{plain}
\newtheorem{thm}{Theorem}
\newtheorem{lem}[thm]{Lemma}
\newtheorem{propo}[thm]{Proposition}
\newtheorem{cor}[thm]{Corollary}

\theoremstyle{definition}

\newtheorem{remark}[thm]{Remark}

\newtheorem{prob}[thm]{Problem}
\renewenvironment{proof}[1][] {\smallskip \noindent {\bf Proof#1.} }{\hspace*{\fill}$\square$\medskip\par}

\def\e{\mbox{e}}
\def\P{{\bf {\mathbb{P}}}}
\def\Q{{\bf {\mathbb{Q}}}}

\def\R{{\bf {\mathbb{R}}}}

\def\s{\mathbf{s}}
\def\S{\mathbf{S}}
\def\dd{\textup{d}}
\def\E{\mathbb{E}}
\def\PP{\mathbb{P}}
\def\EE{\mathbb{E}}

\def\Ba{B^{\textup{a}}}
\def\Xa{X^{\textup{a}}}

\begin{document}

 \title[Co-ascent processes and Palm calculus]{Self-similar co-ascent processes and Palm calculus}
 \author{Christian M\"onch}\thanks{Part of this research was funded by DFG grant no. 443916008}
\address{Institut f\"ur Mathematik\\ Johannes Gutenberg-Universit\"at Mainz\\ Staudingerweg\ 9\\ 55099 Mainz\\ Germany}
\email{cmoench@uni-mainz.de}
 \date{November 25, 2026}
 \begin{abstract}
 We study certain renormalised first passage bridges of self-similar processes, generalising the ``Brownian co-ascent process'' discussed by Panzo (S{\'e}m. Prob. L, 2019) and introduced by Rosenbaum and Yor (S{\'e}m. Prob. XLVI, 2014). We provide a characterisation of co-ascent processes via Palm measures, namely that the co-ascent of a self-similar process is the process under the Palm distribution associated with its record measure. We use this representation to derive a distributional identity for $\alpha$-stable Lévy-subordinators with $\alpha\in(0,1)$.
 \end{abstract}
 \maketitle

\section{Introduction}
\subsection*{Overview} In this note we discuss renormalised first passage bridges, called \emph{co-ascent processes}, derived from self-similar continuous processes. We argue that these co-ascent process can be represented via Palm measures in a natural way for a large class of self-similar processes and that a suitably rescaled co-ascent process is the canonical choice to define the distribution of a ``typical'' sample path under the condition that the process is at its running supremum at a given fixed time.

An important tool in our argumentation is the Palm distribution associated with the record measure of a given process. If the underlying process possesses some shift-invariance property, such as stationary increments, the record measure is usually not well behaved with respect to shifts of the underlying path space. This complicates the use of standard results from Palm theory. It is therefore natural to work with a self-similar process and use rescalings instead of shifts on the path space. A logarithmic change of coordinates then allows us prove our results in the more familiar framework of shifts. The main insight of our investigation can be informally stated as ``The co-ascent process $\Xa$ of a self-similar process $X$ is Palm distributed with respect to the record measure $\mu$ of $X$''. This yields a natural interpretation of co-ascent processes in terms of Palm distributions. We believe that this representation of co-ascent processes is the starting point of their systematic study outside the Markovian setting. As an application of our results, we derive a novel distributional identity for self-similar pure jump processes in Corollaries~\ref{cor:12subordinator} and~\ref{cor:sssubordinator}. 

Let us now briefly discuss the background of our results, first in the context of path transformations of Brownian motion and then in the context of Palm theory for stochastic processes.
\subsection*{Brownian co-ascent and related processes}
As a motivational example, we first consider the Brownian case. Let $B=(B_t)_{t\geq 0}$ be a standard Brownian motion. Recall that $B$ is the unique continuous Gaussian process which satisfies $\EE B_1^2=1$, has stationary increments and is also $\nicefrac{1}{2}$-self-similar, i.e. for any $c>0$,
\begin{equation*}
\left(\frac{1}{\sqrt{c}}B_{ct} \right)_{t\geq 0}\overset{d}{=}(B_t)_{t\geq 0}.
\end{equation*}
where $\overset{d}{=}$ denotes equality in distribution. Let $T_1=\inf\{s>0: B_s>1\}$ denote the first hitting time of level $1$. We call the process $\Ba$ defined by
\begin{equation}\label{eq:fpb1}
\Ba_t:= \frac{1}{\sqrt{T_1}}B_{{t}{T_1}}, \quad t\geq 0,
\end{equation}
the \emph{(extended) co-ascent process} associated to $B$. We use the qualifier `extended' because we consider $\Ba_t$ for all positive times, whereas originally the definition only included the interval $[0,1]$. 
The name Brownian co-ascent for the process $(\Ba_t)_{0\leq t\leq 1}$ was coined recently by Panzo \cite{panzo2018scaled}, who established the relation \[(\Ba_s)_{0\leq s\leq 1}=(m_{1}-m_{1-s})_{0\leq s\leq 1},\] where $(m_s)_{0\leq s\leq 1}$ is the Brownian co-meander, which is obtained by running the absolute value of the Brownian excursion straddling $1$ backwards from its endpoint to time $1$ and rescaling it to unit duration. Hence the Brownian co-ascent process can be related to the Brownian co-meander in the same fashion as the Brownian ascent is related to the Brownian meander. Here, the Brownian ascent is obtained by conditioning $(B_t)_{0\leq t\leq 1}$ to attain its maximum at $1$, whereas the Brownian meander is obtained by running the absolute value of the Brownian excursion straddling $1$ forwards from its starting point to time $1$ and rescaling it to unit duration. Intuitively, the conditioning in the definition of the ascent favours a long period of increase and thus an atypically (for the original process) large running maximum, whereas the co-ascent process produces a running maximum of typical order at time $1$, cf.\ Proposition~\ref{prop:easy} below. We refer the reader to \cite{panzo2018scaled} for a detailed discussion of Brownian (co-)ascent and (co-)meander and a more comprehensive overview of their relations with each other.

Related constructions for Brownian motion have been amply investigated before: let $U\in[0,1]$ be uniform and independent of $B$ and define the random variable 
\[
\alpha= \Ba_U.
\]
The distribution of $\alpha$ was studied by Elie et al. in \cite{elie2014expectation} and shown to appear in many interesting distributional identities for functionals derived from Brownian motion, see also \cite{rosenbaum2014law,rosenbaum2015some}. The random variable $\alpha$ is intimately connected to the {pseudo-Brownian bridge} introduced by Biane et al. in \cite{biane1987processus}, which is formally obtained by replacing the first hitting time $T_1$ in the definition \eqref{eq:fpb1} by the first hitting time of level $1$ of Brownian local time at $0$. Note further that, conditional on $\Ba_1=\lambda$, the co-ascent process is a Brownian first passage bridge to level $\lambda$. This can be seen by expressing the law of the co-ascent through the law of the ascent via the change of measure derived in \cite[Corollary 12.2]{panzo2018scaled}, to obtain that the law of ascent and co-ascent agree when conditioning on the endpoint and then using the results of \cite{bertoin2003path}.  

It is immediate from an application of the strong Markov property to $B$ at the stopping time $T_1$ that $(\Ba_{1+s}-\Ba_1)_{s\geq 0}$ is a Brownian motion independent of $(\Ba_{s})_{0\leq s\leq 1}.$ It is also straightforward to see that $\Ba$ cannot be a self-similar process, because it achieves its running maximum at $t=1$, but any space-time rescaling by a non-trivial factor yields a process that a.s.\ does not have this property. However, we will see below that $\Ba$ is self-similar under rescaling by first hitting times. This is a manifestation of a well known feature of Palm distributions of diffuse random measures called {mass-stationarity}, see \eqref{eq:ms} below.

\subsection*{Palm theory and some of its applications}
Palm calculus was originally developed to study inter-arrival times in point processes \cite{palm1943intensitatsschwankungen}. Later, Mecke \cite{mecke1967stationare} generalised the notion to random measures on locally compact Abelian groups. The idea of applying Palm calculus to random measures (or equivalently additive functionals) derived from stochastic process is also classical, in particular in the study of local times of Markov process, see e.g.\ \cite{geman1973remarks} and the references there and for stationary processes, see e.g.\ \cite{geman1973a}. In the late 1980's, Z\"ahle developed a general method to study fractal properties of a large class of measures obtained from general self-similar processes with stationary increments \cite{Z3}, based on the Palm calculus of self-similar random measures put forward in \cite{Z1,Z2}.

More recently, in \cite{heveling2005characterization,heveling2007point,
	last2009invariant,last2014unbiased,last2015construction},
 Last et al. proved a number of characterisation theorems for Palm measures, some of which we apply in Section 2 in the self-similar setting.

From the point of view of applications in the context of stochastic processes, Palm theory has been used to characterise the distributions of Markovian bridges \cite{MR1278079} and has also proven very fruitful in tackling problems related to embedding distributions (of random variables or random functions) into Brownian paths, see \cite{last2014unbiased, pitman2015slepian,morters2017optimal,last2018transporting}, and also \cite{morters2016skorokhod} for an application in discrete time. For non-Markovian processes, a related technique based on Z\"ahle's approach in \cite{Z3} has been employed in \cite{Moe18} to derive the persistence exponent for local times of self-similar processes with stationary increments.

In all examples above, Palm measures are defined for processes exhibiting some shift-invariance property such as stationarity or stationarity of increments. Since we deal with self-similar processes, we need to perform most of our calculation under a logarithmic rescaling of space and time. Note that we focus on record measures of self-similar processes, which possess (even in the stationary or stationary increment case) no inherent shift-invariance, unlike e.g.\ occupation measures or local times.

Via the first hitting time $T_1$ in \eqref{eq:fpb1}, record measures are related to the co-ascent process. Below, we identify $T_1$ as `typical' record time in the sense advocated by Last and Thorisson in \cite{last2011typical} and intimately connected to Palm measures. In fact, Palm measures are often described intuitively as `having a typical point at the origin', which is the point $1$ in our set up. The co-ascent process is the original process seen from a typical record, or more aptly, seen on the \emph{scale} of a typical record.

We remark that we only treat co-ascents to positive levels, i.e. positive records but all arguments carry over to the case of `descent processes' and negative records by considering $-X$ instead of the process $X$.
%

\subsection*{Outline of the following sections}
We develop the general setup and prove our main results in Section~2. Section~3 is devoted to some concluding remarks and open questions. 
\section{General set up and results}
\subsection*{Co-ascent processes}
Let us now introduce co-ascent processes in a general form. Throughout $X=(X_{t})_{t\geq 0}$ is $H$\emph{-self-similar} for some $H\in(0,1)$, i.e.\ for any $c>0$ we have
\begin{equation}\label{eq:Hss}
(c^{-H}X_{ct})_{t\geq 0}\overset{d}{=} (X_t)_{t\geq 0},
\end{equation}
and also \emph{ergodic}: if $E$ is a measurable set of trajectories with $$\P(\{(X_t)_{t\geq 0}\in E\} \Delta \{c^{-H}X_{ct})_{t\geq 0}\in E\})=0\quad \text{for all c>0},$$ then $\P((X_t)_{t\geq 0}\in E)\in\{0,1\}$. Note that this notion of ergodicity is natural in the context of self-similarity, but for self-similar processes with stationary increments does not coincide with the usual definition, which refers to shifts and not rescalings. For instance, it is known that fractional Brownian motion for any $H\in(0,1)$ is ergodic in this sense \cite{MR0991287}. We further assume that $X$ a.s.\ admits a version with continuous paths, we always identify $X$ with this version. Note that continuity and self-similarity imply that necessarily $X_0=0$ a.s. The running supremum process $M=(M_t)_{t\geq 0}$ with $M_t=\sup_{0\leq s\leq t} X_s$ is assumed to satisfy $\PP(M_1>0)=1$. In this case, \eqref{eq:Hss} implies that $M_{\varepsilon}>0$ a.s.\ for any $\varepsilon>0$ and that $\lim_{t\to\infty} M_t=\infty$ a.s. We set $T_x=\inf\{t> 0: X_t> x\}, x\geq 0$ and note in passing, that $M$ is $H$-self-similar, $T=(T_x)_{x\geq 0}$ is the right-continuous inverse of $M$ and, consequently, $T$ is $\nicefrac1H$-self-similar.
The \emph{(extended) co-ascent process} $\Xa$ of $X$ is given by
\begin{equation}\label{eq:defXa}
\Xa_t= T_1^{-H}X_{T_1 t},\quad t\geq 0.
\end{equation}
To conclude our introduction of $\Xa$, we remark that the choice of the passage time in \eqref{eq:defXa} plays no role.
\begin{lem}\label{prop:levelind}
Fix $x>0$ and set $\Xa_t(x)= T_x^{-H}X_{T_x t}, t\geq 0.$ The corresponding process $\Xa(x)$ is equal in distribution to $\Xa(1)=\Xa$.
\end{lem}
\begin{proof}
This is a direct consequence of the scaling property \eqref{eq:Hss}. Consider the pair $(Y,S)$ given by \[
\left(Y_s  , S_y \right)= \left(\frac1x X_{x^{1/H}s}, \frac{1}{x^{1/H}}T_{yx} \right),\quad s\geq 0, y\geq 0,
\]
and observe that $S_y,y\geq 0$, are precisely the first passage times of the space-time rescaled process $Y$. We have $\smash{(Y,S)\overset{d}{=}(X,T)}$ by self-similarity and consequently $\smash{\Xa\overset{d}{=}Y^{\textup{a}}}$, whilst at the same time
\[
Y^{\textup{a}}_s= S_1^{-H}Y_{S_1s}= x T_x^{-H}\; x^{-1} X_{x^{1/H}s\,x^{-1/H}T_x}=\Xa_s(x),\quad s\geq 0,
\]
which concludes the proof.
\end{proof}

$(\Xa_t)_{0\leq t\leq 1}$ can be interpreted as the rescaled ascension of $X$ to a `typical level': the rescaling removes information about the specific choice of the level $x$ in the sense that the original path of $X$ can be recovered from $\Xa$ given $x$, but not without the knowledge of $x$. Although $\Xa$ is not $H$-self-similar, it is $H$-self-similar under rescaling by first passage times.
\begin{thm}\label{thm:projection}
	Let $X$ be a continuous $H$-self-similar ergodic process with running supremum $M$ where $\E M_1<\infty$. Then \[
	(\Xa)^{\textup{a}}\overset{d}{=}\Xa.
	\]	
\end{thm}
Theorem~\ref{thm:projection} is derived below as a consequence of Corollary~\ref{thm:ascispalm}. Note that even in the Brownian case, Theorem~\ref{thm:projection} has non-trivial consequences.
\begin{cor}\label{cor:12subordinator}
Let $(S_{x})_{x\geq 0}$ denote the $\nicefrac{1}{2}$-stable Lévy subordinator normalised such that $S_1\overset{d}{=}\inf\{s>0: B_s>1\}$, where $B$ is standard Brownian motion. Then we have the distributional identity
\[
\frac{1}{S_1}S_{{S_1^{1/2}}}\overset{d}{=}S_1.
\]
\end{cor}
Corollary~\ref{cor:12subordinator} is a special case of the following more general observation.
\begin{cor}\label{cor:sssubordinator}
	If $(S_{x})_{x\geq 0}$ is a strictly increasing $\nicefrac{1}{H}$-self-similar ergodic pure jump process with $\EE \inf\{{x>0:S_x>1}\}<\infty$, then
	\[
	\frac{1}{S_x}S_{x{S_x^{H}}}\overset{d}{=}S_x,\quad x>0.
	\]
\end{cor}
\begin{proof}
For ease of notation, we give the argument for $x=1$. For general $x$, the same calculation applied to the process $Y^{\textup{a}}(x)$ defined in Lemma~\ref{prop:levelind} yields the result. Note that since $S$ is pure jump and monotone, its right continuous inverse $Y=S^{-1}$ is well-defined, continuous and monotone. The assumption $\EE \inf\{{x>0:S_x>1}\}<\infty$ implies that $Y$ has finite expectation. Moreover, the inversion map $$S\overset{(\cdot)^{-1}}\mapsto Y$$ turns $\nicefrac{1}{H}$-self-similarity into $H$-self-similarity and preserves ergodicity (under the respective rescalings). By monotonicity, we have $Y^{\textup{a}}_1=S_1^{-H}$, whilst \begin{align*}
(Y^{\textup{a}})^{\textup{a}}_1 & = \inf\{ t: Y^{\textup{a}}_t>1 \}^{-H} =  \inf\{ t: Y_{tS_1}>S_1^{H} \}^{-H} \\ & = S^{H}_1\inf\{ s: Y_{s}>S_1^{H} \}^{-H},
\end{align*}
hence by applying Theorem~\ref{thm:projection} to $Y$
\[
S_1^{-H}\overset{d}{=}\left(\frac{1}{S_1}S_{S^{H}_1} \right)^{-H},
\]
and the assertion follows.
\end{proof}
\begin{proof}[ of Corollary~\ref{cor:12subordinator}]
We verify the assumptions of Corollary~\ref{cor:sssubordinator}: that the running supremum of Brownian motion has finite expectation and that its inverse $S$ is the $\nicefrac{1}{2}$-stable subordinator are basic facts about Brownian motion, see e.g.\ \cite{MortersBook}. The only non-standard result needed is ergodicity of Brownian motion under rescaling, which follows e.g.\ from \cite[Example 1]{MR0991287} and implies ergodicity of $S$.
\end{proof}
%
%

\subsection*{Processes and measures with invariance properties}
We now take a more abstract point of view that does not only take the process $X$ into account but also associated measures. For technical reasons, we switch between self-similar and stationary processes, thus we introduce both settings. Let $(\Omega,\mathcal{A})$ denote a measurable space and $((G,\ast),\mathcal{G})$ a measurable group that acts measurably on $\Omega$, i.e.\ $\omega\mapsto g\omega$ is $\mathcal{A}$-measurable for any fixed $g\in G$ and $(\omega,g)\mapsto g\omega$ is $\mathcal{A}\otimes\mathcal{G}$-measurable. A measure $\P$ on $(\Omega,\mathcal{A})$ is called \emph{invariant}, if $\P\circ g=\P$ for all $g\in G$. Let now $(\hat\Omega,\hat{\mathcal{A}})$ denote another measure space on which a measurable group $(\hat{G},\bullet)$ acts measurably and assume that there is an injective morphism $\smash{\hat{G}\overset{i}\hookrightarrow G}$. The random variable $X:\Omega\to\hat{\Omega}$ is called $\hat{G}$-\emph{covariant}, if
\[
gX(\omega)=X(i(g)\omega),\quad g\in\hat{G},\omega\in\Omega,
\]
or, more concisely,
\[
g \circ X = X \circ i(g), \quad g\in\hat{G}.
\]
Throughout, we assume that $G$ has a subgroup $G'$ which is isomorphic to $(\R,+)$, which we represent as $G'=(g_s,s\in\R)$ with $g_s\ast g_t=g_{s+t}$ for all $s,t\in\R$. Let $\mathcal{C}_0(\R_{\geq 0},\R)$ denote the space of all continuous functions $f$ on $\R_{\geq 0}$ satisfying $f(0)=0$, equipped with the topology of uniform convergence and the corresponding Borel-$\sigma$-field $\mathcal{F}_0$. For fixed $H\in(0,1)$ we consider the group $\S(H)=(\s_r)_{r>0}$ of scaling maps acting on $\mathcal{C}_0(\R_{\geq 0},\R)$ via
\[
(\s_r f)_t:= r^{H}f_{t/r},\quad t\geq 0, r>0, f\in\mathcal{C}_0(\R_{\geq 0},\R).
\]
Here and throughout, we usually write $f_s,$ instead $f(s)$ for functions defined on $\R$ or $\R_{\geq 0}$, keeping to our notation for stochastic processes. Any $H$-self-similar process can now be obtained via a corresponding $\S(H)$-covariant random variable $X:\Omega\to\mathcal{C}_0(\R_{\geq0},\R)$ from an invariant probability measure $\P$ on $\Omega$. Note that for this identification we may and shall always use the injection $\s_r\mapsto g_{\log r},r>0$. The corresponding distribution $\P\circ X^{-1}$ on $\mathcal{C}_0(\R_{\geq0},\R)$ is then $H$-\emph{scale-invariant}, since for any $r>0$,
\[
\P(\s_rX(\omega)\in F) =\P(X(g_{\log r}\omega)\in F)= \P(X(\omega)\in F),\quad F\in\mathcal{F}_0.
\]
Similarly, we obtain any stationary process via a corresponding $\Theta$-covariant map $\hat{X}:\Omega\to\mathcal{C}(\R,\R)$, where $\mathcal{C}(\R,\R)$ is the space of all continuous functions equipped with the topology of uniform convergence and the corresponding Borel-$\sigma$-field $\mathcal{F}$. Here $\Theta=(\theta_t)_{t\in\R}$ is the group of shifts acting on paths via
\[
(\theta_tf)_s=f_{s+t},\quad s\in\R,t\in\R,f\in\mathcal{C}(\R,\R),
\]
and we canonically embed $\Theta$ into $G$ through $\theta_t\mapsto g_t$. Define now, for fixed $H\in (0,1)$, the map $L_H:\mathcal{C}_0(\R_{\geq 0},\R)\to\mathcal{C}(\R,\R)$ via
\[
L_H f_t=\textup{e}^{-Ht} f_{e^{t}},\quad t\in\R.
\]
It is well known, see \cite{lamperti1962semi}, (and straightforward to check) that $\hat{\P}$ on $\mathcal{C}(\R,\R)$ is \emph{shift-invariant} if and only if $\hat{\P}=\Q\circ L_H^{-1}$ for some $H$-scale invariant $\Q$ on $\mathcal{C}_0(\R_{\geq 0},\R)$. In particular, $L_H$ induces a 1-to-1 correspondence between stationary processes and self-similar processes, if we define its inverse as
\[
L_H^{-1}f_0=0,\; L_H^{-1}f_t=t^H f_{\log t},\quad t>0, f\in\mathcal{C}(\R,\R).
\]
To save notation, we write for the remainder of the paper $\hat{X}=L_H\circ X$, i.e.\ if $X$ is an $H$-self-similar process, then $\hat{X}$ is its stationary \emph{Lamperti representation}. It is straightforward to see, that $X$ is ergodic under recaling if and only if $\hat{X}$ is ergodic with respect to shifts.

Let now $\mathcal{M}$ denote the space of all locally finite measures over the real Borel-$\sigma$-field $\mathcal{B}(\R)$ and let $\mathcal{M}_0$ denote the space of all locally finite measures $m$ on $\mathcal{B}(\R_{\geq 0})$ with $m(\{0\})=0$. Both $\mathcal{M}$ and $\mathcal{M}_0$ are equipped with their Borel-$\sigma$-fields. Random variables $\xi:\Omega\to \mathcal{M}$, $\xi^0:\Omega\to \mathcal{M}_0$ are \emph{random measures} and we denote their associated additive functionals (or distribution functions) by
\[
\xi^0_t =\xi^0((0,t]), t\geq 0,\quad \text{and}\quad \xi_t= \begin{cases} \xi((0,t]), & t\geq 0\\ -\xi((t,0]), & t<0.\end{cases}
\]
The measure defined by $\E\int \ind_A(x) \xi(\textup{d} x)$, $A\in\mathcal{B}(\R)$, is called the \emph{intensity measure} of $\xi$; an analogous definition applies to $\xi^0$. We fit random measures into the framework above by saying that $\xi^0$ is $H$-self-similar, if it is $\S(H)$-covariant and that $\xi$ is stationary, if is $\Theta$-covariant. Here, the actions of $\S(H)$ and $\Theta$ on measures are formally defined via their actions on the corresponding additive functional, but it is straightforward to verify that this corresponds to the usual definition of stationary/self-similar random measure. Note, however, that the Lamperti representation $L_H$ does not apply directly to random measures via their additive functionals, since $L^{-1}_H$ does not map monotone functions to monotone functions. Hence, in case of a self-similar random measure $\xi$, we reserve the notation $(\hat\xi_t)_{t\in\R}$ for the additive functional of the measure $\hat\xi$ given by
 \begin{equation}\label{eq:Lamperti-Def}
\hat\xi (C):=\bar{L}_H \xi(C):= \int_0^\infty \ind_C(\log x) x^{-H}\xi (\dd x),\quad C\in \mathcal{B}(\R).
\end{equation}
Indeed, it is not difficult to show, that the map $\bar{L}_H$ thus defined extends the Lamperti representation to measures, and in particular provides a one-to-one correspondence between self-similar random measures on $\R_{\geq 0}$ and stationary random measures on $\R$, cf.\ \cite[1.3]{Z1}.
 
Our set up so far is summarised in the following diagram:

\begin{equation*}
\adjustbox{scale=0.9,center}{%
	\begin{tikzcd}
	&&& \Omega \\
\\
\\
{\mathcal{C}_0(\R_{\geq 0},\R)\times \mathcal{M}_0} &&&&&& {\mathcal{C}(\R,\R)\times\mathcal{M}}
\arrow["{(X,\xi)}"{description}, curve={height=24pt}, from=1-4, to=4-1]
\arrow["{(\hat{X},\hat{\xi})}"{description}, curve={height=-24pt}, from=1-4, to=4-7]
\arrow["{(L_H,\bar{L}_H)}"{description}, curve={height=12pt}, from=4-1, to=4-7]
\arrow["{(L_H^{-1},\bar{L}_H^{-1})}"{description}, curve={height=12pt}, from=4-7, to=4-1]
	\end{tikzcd}
}
\end{equation*}
The identification via $L_H$ allows us to infer all necessary results about self-similar processes and measures straightforwardly from their stationary counterparts.
\begin{remark}
	Of course, $\S(H)$ can be identified with $(\R_{>0},\cdot)$ and $\Theta$ with $(\R,+)$ and we could let either group act directly on the path spaces via appropriate definitions (for $H$ fixed). Furthermore, an equivalent approach is to simply set $(\Omega,G)=(\mathcal{C}(\R,\R)\times\mathcal{M},\Theta)$ and further $(\hat{\Omega},\hat{G})=(\mathcal{C}_0(\R_{\geq 0},\R)\times\mathcal{M}_0,\S(H))$ and use the push forward $(L_H^{-1},\bar{L}_H^{-1},\exp)$ for function, measure and group to infer results on self-similar random variables from stationary ones. We chose the above set up with the abstract space $\Omega$ and the $\R$-isomorphic subgroup $G'$ to emphasise the symmetry between the stationary and self-similar world.
\end{remark}

\subsection*{Random measures and Palm distributions}
We now work with a fixed invariant `reference' measure $\P$ (note that $\P$ is assumed to be $\sigma$-finite but does not need to be a probability distribution) on $\Omega$ and denote the corresponding integral/expectation by $\E$. Let $(Z,\zeta)$ denote a stationary pair, i.e. the process $Z=(Z_t)_{t\in\mathbb{R}}$ and the random measure $\zeta\in\mathcal{M}$ are shift-covariant maps. We say $(Z^\circ,\zeta^\circ)$ is a \emph{Palm version} of $(Z,\zeta)$, if for all non-negative measurable functions $h$ and all compact $A\subset \R$ of positive Lebesgue measure $\lambda(A)>0$,
\begin{equation}\label{eq:refinedCamp}
{\EE}(h(Z^\circ,\zeta^\circ))={\EE}\left[
\int_A h(\theta_{-r}(Z,\zeta))\zeta (\dd r)
\right] \lambda(A)^{-1}.
\end{equation}
 Note that here and in the self-similar case, we always interpret the action of a group element such as $\theta_{-r}$ on the pair $(Z,\zeta)$ component-wise, i.e.\ $$\theta_{-r}(Y,\zeta)=(\theta_{-r}Y,\theta_{-r}\zeta).$$ 
 \begin{thm}\label{thm:ascispalmweak}
 Let $X$ be an $H$-self-similar ergodic process with supremum process $M$, where $\E M_1<\infty$. Let $\mu$ denote the record measure of $X$, i.e.\ $\mu_t=M_t$, $t\geq 0$. Define a measure $\mu^{\textup{a}}$ via $\mu^{\textup{a}}_t=M^\textup{a}_t$. Then $(\hat{\Xa},\hat{\mu^\textup{a}})$ is a Palm version of $(\hat{X},\hat{\mu})$.
 \end{thm}
\begin{remark}\label{remk1}\begin{enumerate}[(a)]
\item Forming the running supremum of $X$ and forming the co-ascent are commutative operations, since the first passage times of $X$ to positive levels coincide with those of its running supremum $M$ and forming the running supremum commutes with the action of $\S(H)$. In other words, the co-ascent of the running supremum process coincides with the running supremum of the co-ascent process.
\item The Lamperti-representation does not commute with forming the running supremum. In particular, $\hat{\mu}$ is not the record measure of $\hat{X}$, but instead a stationary diffuse random measure with support $$\bigcup_{z>0}\big\{\inf\{t\in\R: \hat{X}_t=\e^{-Ht}z\}\big\}\subset \R.$$
\end{enumerate}
\end{remark}
  
To aid in the proof of Theorem~\ref{thm:ascispalmweak} below, we introduce another tool. One can interpret the relation between the random pairs $(Z,\zeta)$ and $(Z^\circ,\zeta^\circ)$ in $\eqref{eq:refinedCamp}$ under $\PP$ as a change of measure formula. The measure $\Q_\zeta$ satisfying
\begin{equation}\label{eq:Palm}
\int h \,\dd \Q_\zeta =\EE\left[
\int_A h\circ \theta_{-r}\,\zeta (\dd r)
\right] \lambda(A)^{-1},
\end{equation}
for given $A$ as above and any measurable $h:\Omega\to\R_{\geq 0}$ is called the \emph{Palm measure} of $\P$ with respect to $\zeta$. The measure $\Q_\zeta$ is not necessarily a probability measure, but it is easily seen that by stationarity the right hand side of \eqref{eq:Palm} does not depend on the choice of $A$ and $\Q_\zeta$ is unique up to multiplication by a constant. If $\Q_\zeta(\Omega)$ is finite, then $\PP^\circ_\zeta(\cdot) = \Q_\zeta(\Omega)^{-1}\Q_\zeta(\cdot)$ is called the \emph{Palm distribution} of $\PP$ with respect to $\zeta$, its associated expectation is denoted by $\EE^\circ_\zeta.$ Equivalently, if $\E \zeta_1<\infty$ and $A$ in \eqref{eq:Palm} has unit length, then $\PP^\circ_\zeta=(\EE \zeta_1)^{-1}\Q_\zeta$.

\begin{proof}[{ of Theorem~\ref{thm:ascispalmweak}}]
It suffices to show that $(\hat{\Xa},\hat{\mu^{\textup{a}}})$ is Palm distributed with respect to $\hat\mu$. Note that a Palm distribution exists, since the intensity measure of $\mu$ has Lebesgue density $H\E M_1 x^{H-1}$ due to self-similarity and thus $\hat\mu$ has finite intensity $H\E M_1$. Let $A\in\mathcal{B}(\mathcal{C}(\R,\R))\otimes\mathcal{B}(\mathcal{M})$ be an arbitrary event. Under the Palm distribution with respect to $\hat\mu$, the probability of $A$ can be written as 
\begin{align}\label{eq:Jz1}
\P^{\circ}_{\hat\mu}(A)& = \E\left[ \frac{\int_0^t\ind_A\circ\theta_{-s}\,\hat{\mu}(\dd s)}{t} \right]\,\frac{1}{\E \hat{\mu}_1} =\E\left[ \frac{\int_0^{\hat{\mu}_t} \ind_A\circ\theta_{-\hat{\mu}^{-1}_r}\,\dd r}{t} \right]\,\frac{1}{\E \hat{\mu}_1}
\end{align}
for arbitrary $t>0$. Here, $(\hat{\mu}^{-1}_r)_{r\in\R}$ denotes the right-continuous inverse of the additive functional $(\hat{\mu}_r)_{r\in\R}$ belonging to the measure $\hat\mu$. Set 
\[
J(z)=\int_0^{z} \ind_A\circ\theta_{-\hat{\mu}^{-1}_r}\,\dd r, \quad z>0.
\]
Assume for the moment, that $\EE[J(z\E\hat{\mu}_1)/z]\equiv \tilde J$ does not depend on $z$. Now we fix $\delta>0$ and define a family of events $\{{E}(t,\delta), t>0\}$ via \[{E}(t,\delta) = \Big\{1-\delta \leq \frac{\hat{\mu}_t}{t\E\hat{\mu}_1}\leq 1+\delta \Big\}, \quad t>0.\] We claim that 
\begin{equation}\label{claim1}
\P^{\circ}_{\hat\mu}(A)\E\hat{\mu}_1=\lim_{t\to\infty}\E\Big[\frac{J(\hat{\mu}_t)}{t}\ind_{{E}(t,\delta)} \Big],
\end{equation}
and
\begin{equation}\label{claim2}
\tilde{J}=\lim_{t\to\infty}\E\Big[\frac{J(t\E\hat\mu_1)}{t}\ind_{{E}(t,\delta)} \Big].
\end{equation}
To see that the claims are true, recall that $\hat\mu$ is ergodic under shifts and hence the averages $t^{-1}\hat{\mu}_t$ converge a.s.\ to $\E\hat{\mu}_1$ by Birkhoff's Ergodic Theorem. Since $J(z)\leq z$ for all $z>0$, we have
\[
\max\Big\{\frac{J(\hat{\mu}_t)}{t}\ind_{{E}(t,\delta)}, \frac{J(t\E\hat\mu_1)}{t}\ind_{{E}(t,\delta)}\Big\}\leq \E\hat{\mu}_1(1+\delta),
\]
and an application of dominated convergence yields the equalities \eqref{claim1} and \eqref{claim2}. Further, we have for all $t>0$
\[
(1-\delta)\E\Big[\frac{J(t\E\hat{\mu}_1(1-\delta))}{(1-\delta)t}\ind_{{E}(t,\delta)} \Big]\leq \E\Big[\frac{J(\hat{\mu}_t)}{t}\ind_{{E}(t,\delta)} \Big]\leq (1+\delta)\tilde{J},
\]
by monotonicity of $J$. In the limit $t\to\infty$ we thus obtain
\[
(1-\delta)\frac{\tilde{J}}{\E\hat{\mu}_1}\leq \P^{\circ}_{\hat\mu}(A) \leq (1+\delta)\frac{\tilde{J}}{\E\hat{\mu}_1},
\]
and since $\delta$ was arbitrary, we conclude that $\tilde{J}/\E{\hat{\mu}_1}=\P^{\circ}_{\hat\mu}(A).$ It remains to inspect the expectation of $J(z)/z$ and show that it is independent of $z$. More precisely, we show that
\begin{equation}\label{eq:finalstep}
\E\left[ \frac{\int_0^{z} \ind_A\circ\theta_{-\hat{\mu}^{-1}_r}\,\dd r}{z} \right]=\E[\ind_A\circ\theta_{-\hat{\mu}^{-1}_0}].
\end{equation}
To this end, we denote $\tilde T_r=\hat{\mu}^{-1}_r, r\in\R$, recall that $(T_x)_{x>0}=(\mu^{-1}_x)_{x>0}$ are the first passage times of $X$, and note that these two families of random times satisfy the following correspondence as sets
\[
\{\tilde T_r: r\in\R\} = \{\log (T_x): x>0\}.
\]
Indeed, it follows from the definition \eqref{eq:Lamperti-Def} of the Lamperti representation for measures that a time $t>0$ is a first passage time of $X$ if and only if $\log(t)$ is a first passage time for $(\hat{\mu}_t)_{t\in\R}$. By Lemma~\ref{prop:levelind}, for every $x>0$, the process-measure pair
$\s_{1/T_x}(X,\mu)$ is identically distributed to $(\Xa,\mu^{\textup{a}})$, c.f.\ Remark~\ref{remk1}(a). For each $x,y>0$, we thus have that $\theta_{-\log (T_x)}(\hat X,\hat \mu)$ and $\theta_{-\log (T_y)}(\hat X,\hat \mu)$ have the same distribution and therefore, for each $s,r\in\R$, $\theta_{-\tilde T_s}(\hat X,\hat \mu)$ and $\theta_{-\tilde T_r}(\hat X,\hat \mu)$ have the same distribution. We conclude that
\eqref{eq:finalstep} holds and the proof is complete.
\end{proof}

 A shift-covariant random measure $\zeta$ is called \emph{mass-stationary} under $\P$, if for any bounded Borel set $C$ with $\lambda(C)>0$ and $\lambda(\partial C)=0$ and all measurable $h:\Omega\times\R \to\R_{\geq 0}$
 \begin{equation}\label{eq:ms}
 	\begin{aligned}
	\int_\Omega \int_C\frac{\int_{C-u} h(g_s\omega,s+u)\,\zeta(\omega,\dd s)}{\zeta(\omega,C-u)}& \,\dd u\,\P(\dd \omega)\\ 
	& = \int_\Omega\int_C h(\omega,u)\,\dd u\,\P(\dd\omega).
 	\end{aligned}
 \end{equation}
It is well known that mass-stationarity and shift-invariance with respect to hitting times of $(\zeta_t)$ characterise Palm measures.
 \begin{propo}[{\cite[Theorem 3.1]{last2014unbiased}, cf.\ \cite{geman1973remarks}}]\label{prop:shift-palm}
 Suppose $\zeta$ denotes a shift-covariant random measure on $(\Omega,\mathcal{A})$ and let $(\zeta^{-1}_x)_{x\in\R}$ be the right-continuous inverse of $(\zeta_t)_{t\in\R}$. Let $\Q$ be any measure on $\mathcal{A}$ such that $\zeta(\omega)$ is a non-trivial diffuse random measure for $\Q$-a.e.\ $\omega$. The following statements are equivalent:
 \begin{enumerate}[a)]
 	\item $\Q$ is the Palm measure of some invariant measure $\P$ with respect to $\zeta$.
 	\item We have \[
 	\Q \circ g_{\zeta^{-1}_x} = \Q,\quad x\in \R.
 	\]
 	\item $\zeta$ is mass-stationary under $\Q$.
 \end{enumerate}
 \end{propo}
 Via Lamperti representation, we now immediately obtain an equivalent characterisation result in terms of rescalings instead of shifts.
  \begin{propo}\label{prop:scale-palm}
 	Fix $H\in(0,1)$ and let $\xi$ denote a $\S(H)$-covariant random measure on $(\Omega,\mathcal{A})$ with right-continuous inverse $(\xi^{-1}_x)_{x\geq 0}$. Let $\Q$ be any measure on $\mathcal{A}$ such that $\xi(\omega)$ is a non-trivial diffuse random measure for $\Q$-a.e.\ $\omega$. The following statements are equivalent:
 	\begin{enumerate}[a)]
 		\item $\Q$ is the Palm measure of some invariant measure $\PP$ with respect to $\hat\xi$.
 		\item We have \[
 		\Q \circ g_{\log(\xi^{-1}_x)} = \Q \circ g_{\hat{\xi}^{-1}_x} =\Q \circ g_{\hat{\xi}^{-1}_{-x}} = \Q,\quad x>0.
 		\]
 		\item For any bounded Borel set $A\subset(0,\infty)$ with $\lambda(A)>0$ and $\lambda(\partial A)=0$ and all measurable $\ell:\Omega\times\R \to\R_{\geq 0}$
 		\begin{equation}\label{eq:massscale}
 			\begin{aligned}
 				\int_\Omega \int_A & \frac{\int_{Av^{-1}} \ell(g_{\log s}\omega, \log (sv))\,s^{-H}\xi(\omega,\dd s)}{v\,\int_{Av^{-1}}s^{-H}\xi(\omega,\dd s)} \,\dd v\,\Q(\dd \omega)\\ 
 				& = \int_\Omega\int_A \ell(\omega,\log v) v^{-1}\,\dd v\,\Q(\dd\omega).
 			\end{aligned}
 		\end{equation}
 	\item For any $B\in\mathcal{A}$ and fixed compact $A\subset (0,\infty)$ with $\lambda(A)>0$ and $\lambda(\partial A)=0$ we have 
 			\[
 			\Q(B)=\frac{\int_\Omega \int_A \ind_B (g_{\log s})s^{-H}\xi(\omega,\dd s)\,\dd v \,\P(\dd\omega) }{\int_A s^{-1}\,\dd s},
 			\]
 			for some invariant measure $\P$.
 	\end{enumerate}
 \end{propo}
 \begin{proof}	
 The equivalence of a) and b) is immediate from Proposition~\ref{prop:shift-palm} applied to the stationary random measure $\hat\xi$ (note that shift-covariance of $\hat\xi$ follows from $\bar{L}_H\circ\s_r=\theta_{\log r}\circ \bar{L}_H$) and noting that, for $x>0$, 
 \begin{align*}
 \hat{\xi}^{-1}_x & =\inf\Big\{t:\int_{1}^{\textup{e}^{t}}s^{-H}\xi(\dd s)>x\Big\}\\
 & = \log\Big(\inf\Big\{u:\int_{1}^{u}s^{-H}\xi(\dd s)>x\Big\}\Big)\\
 &= \log\big(\xi^{-1}_y\big)
 \end{align*}
for some $y=y(x)>0$, since the support of $s^{-H}\xi$ coincides with the support of $\xi$. The equivalence of a) and c) is simply the mass-stationarity~\eqref{eq:ms} of $\hat{\xi}$ expressed in terms of $\xi$: it is straightforward to see that, for any measurable non-negative function $f$, $\int f(s)\hat{\xi}(\dd s)=\int \hat{f}(s)s^{-H}{\xi}(\dd s)$ for $\hat{f}=f\circ\log$. Hence, setting $C=\{\log(x), x\in A\}$ and $h=\ell\circ\log$ the formula \eqref{eq:massscale} follows from \eqref{eq:ms} by substitution. Using the same argument for indicator functions, one obtains d) from the definition \eqref{eq:Palm} of the Palm measure with respect to $\hat{\xi}$ in terms of $\xi$.
 \end{proof}
In fact, if $\Q$ satisfies any of the equivalent conditions in Propositions~\ref{prop:shift-palm} or~\ref{prop:scale-palm}, then $\P$ is uniquely determined by $\Q$ \cite[Thm. 3.1]{last2014unbiased}. 
\begin{remark}
\begin{enumerate}[a)]
\item The renormalisation term $\int_A s^{-1}\,\dd s$ in Proposition~\ref{prop:scale-palm} d) reflects the fact that the Haar measure on the multiplicative group $(\R_{>0},\cdot)$ (which we may identify with $\S(H)$ for given $H$) has Lebesgue density $s^{-1}$ and that the intensity measure of any $H$-self-similar measure $\xi$ is necessarily proportional to $s^{H-1}\dd s$. The additional density term $s^{-H}$ in c) and d) accounts for the spatial rescaling, which is not accounted for if we let $(\R_{>0},\cdot)$ act directly on random measures via its canonical action on $\R$ (as can be done for the shift-group $\Theta$).
\item The formulae given in parts c) and d) of Proposition~\ref{prop:scale-palm} are easy to obtain and we make no explicit use of them here, nevertheless we have included them in the statement, since we found no account of formulations of mass-stationarity and the Palm measure under rescaling in the literature. Both concepts are usually either discussed in terms of shifts or completely general in terms of a non-specified group action. When performing calculations involving (rescaling) Palm version of self-similar processes, it is a matter of taste whether one uses c) and d) ad hoc, or instead invokes the Lamperti-representation to work in the more familiar setting of shifts. 
\end{enumerate}
\end{remark} 
The representation of the co-ascent as a Palm distribution is now a consequence of Proposition~\ref{prop:scale-palm}.
Theorem~\ref{thm:ascispalmweak} provides an indirect characterisation of $\Xa$ as the image under $L_H^{-1}$ of a Palm distributed process. We now extend this result by using Proposition~\ref{prop:scale-palm} to establish a direct characterisation. 
\begin{cor}\label{thm:ascispalm}
Let $\P$ be invariant and ergodic and let $X:\Omega\to\mathcal{C}_0(\R_{\geq 0},\R)$ denote a $\S(H)$-covariant map, thus defining an ergodic self-similar process of index $H$. Let $\mu$ denote the record measure of $X$ and let $T_x=\mu_x^{-1}, x\in\R_{\geq 0}$ denote the corresponding first passage times. If $\mu$ has finite intensity, then $\P\circ g_{\log T_1}$ is the Palm distribution of $\P$ with respect to $\mu$.
\end{cor}
\begin{proof}We have
	\[
	\P\circ g_{\log T_1}=\P\circ g_{\hat{\mu}^{-1}_0}=\Q_{\hat{\mu}},
	\]
	where the last inequality follows from the proof of Theorem~\ref{thm:ascispalmweak}.
\end{proof}
We conclude this section by deducing Theorem~\ref{thm:projection}.

\begin{proof}[ of Theorem~\ref{thm:projection}]
Corollary~\ref{thm:ascispalm} identifies $\Xa$ as Palm distributed and Proposition~\ref{prop:scale-palm} b) now entails that the Palm distribution is invariant under rescalings by first passage times.
\end{proof}

\section{Remarks on related problems}
We conclude our considerations with discussing two related open problems.
\subsection*{Two-sided processes with stationary increments}
In \cite{Moe18}, mass-stationarity was used to derive the strong asymptotics of the quantity $\PP(\ell((0,t])\leq 1)$ as $t\to\infty$, where $\ell$ is the local time measure at $0$ of an $H$-self-similar process $X$ with stationary increments. The following related problem still remains open, see also \cite{molchan1999maximum,aurzada2016persistence,aurzada2018persistence}: 
\begin{prob}\label{prob1}
Let $(X_{t})_{t\in\R}$ be a two-sided continuous $H$-self-similar process with stationary increments satisfying $\EE\sup_{0\leq s\leq 1}X_s<\infty$. Can we obtain the strong order of $\PP(X_s \leq 1, 0\leq s \leq t)$? More precisely, does there exist a constant $c_X$, satisfying
\[
\lim_{t\to\infty}\PP(X_s \leq 1, 0\leq s \leq t) t^{1-H}=c_X,
\]
and if so, how can $c_X$ be characterised in terms of $X$?
\end{prob}
In the special cases where $X$ is fractional Brownian motion with Hurst index $H\in (\nicefrac{1}{2},1)$ \cite{aurzada2018persistence} and where $X$ is the Rosenblatt process \cite{aurzada2016persistence} the upper bound $\PP(X_s \leq 1, 0\leq s \leq t) t^{1-H}\leq C$ is known to hold. Note that to calculate the persistence probabilities one may work with the co-ascent instead of the original process.
\begin{propo}\label{prop:easy}
	Let $\Xa$ denote the co-ascent process of some continuous $H$-self-similar process $X$. Then, for any $x>0$,
	\[
	\PP(\Xa_t \leq x, 0\leq t\leq 1)=\PP(X_t\leq x, 0\leq t\leq 1)
	\]
\end{propo}
\begin{proof}
	By definition, the co-ascent process on $[0,1]$ reaches its maximum at $1$. Hence,
	\[
	\PP(\Xa_t \leq x, 0\leq t\leq 1)=\PP(\Xa_1 \leq x)=\PP(X_{T_1}\leq x T^H_1).
	\]
	But $X_{T_1}=1$ a.s. and thus
	\[
	\PP(\Xa_t \leq x, 0\leq t\leq 1)=\PP(T_1\geq x^{-\frac{1}{H}})=\PP(T_x\geq 1)=\PP(M_1\leq x).
	\]
\end{proof}

Understanding $\PP(X_s \leq 1, 0\leq s \leq t)$ is a step towards defining a version of $X$ conditioned on not returning to $0$. The (reversed) co-ascent process can be interpreted as a natural choice for a process derived from (the one sided process) $X$ which does not return to a given level. However, if there was a probability measure on paths that is mass-stationary (in the ordinary sense, i.e.\ under shifts) for the (one-sided) record time measure, then the corresponding (Palm distributed) process is a natural way to induce a change of measure under which paths do never return to a given level. Unfortunately, it is not clear whether such a process exists in general.
\begin{prob}\label{prob2}
	Let $(X_{t})_{t\in\R}$ be a two-sided continuous $H$-self-similar process with stationary increments satisfying $\EE\sup_{0\leq s\leq 1}X_s<\infty$. Is there a two-sided process $X^{\textup m}$ derived from $X$ in a natural way satisfying 
	\[
	(X^{\textup m}_{T_x+t}-x)_{t\in \R} \overset{d}{=}(X^{\textup m}_t)_{t\in\R},
	\]
	where $T_x$ denotes the first hitting time of level $x$ after $0$ of $X^{\textup m}$? 
\end{prob}

\subsection*{Acknowledgment}
I thank Steffen Dereich and Zachar Kabluchko for an interesting exchange at the University of M\"unster about how to employ mass-stationarity to tackle Problems~\ref{prob1} and~\ref{prob2}, which motivated this research; Frank Aurzada and Lisa Hartung for valuable discussions and comments on early drafts of this manuscript, and Hugo Panzo and Martin Kilian for pointing out several mistakes in an earlier version. Furthermore, I would like to thank an anonymous referee who provided some very useful pointers to existing literature on Brownian path transformations and the explicit relation between Brownian co-ascent and first passage bridges discussed in the introduction.

\bibliography{SSSI} 
\bibliographystyle{plain}

\end{document}